\documentclass[12pt]{amsart}

\usepackage{amssymb}
\usepackage{tikz}
\usepackage{xcolor}
\usepackage[margin=1.5in]{geometry}
\usepackage{comment}
\usepackage{amsrefs}
\usepackage{kbordermatrix}
\renewcommand{\kbldelim}{(}
\renewcommand{\kbrdelim}{)}
\usepackage{multirow}
\usepackage{enumitem}
\usepackage{microtype}

\newtheorem{lemma}{Lemma}[section]
\newtheorem{theorem}[lemma]{Theorem}

\newtheorem{proposition}[lemma]{Proposition}
\theoremstyle{definition}
\newtheorem{definition}[lemma]{Definition}

\newtheorem{claim}[lemma]{Claim}
\newtheorem{example}[lemma]{Example}
\newtheorem{remark}[lemma]{Remark}
\newtheorem{property}[lemma]{Property}

\newtheorem{notation}[lemma]{Notation}

\theoremstyle{remark}

\newcommand{\abs}[1]{\left\lvert#1\right\rvert}
\renewcommand{\qq}[1]{\mathrm{mon}_#1}
\newcommand{\QQ}[1]{\mathrm{Mon}_#1}

\newcommand{\A}{\mathbb{A}}

\makeatletter
\newcommand{\superimpose}[2]{{\ooalign{$#1\@firstoftwo#2$\cr\hfil$#1\@secondoftwo#2$\hfil\cr}}}
\makeatother
\newcommand{\ttimes}{\hspace{0.4mm}{\mathpalette\superimpose{{\circ}{\cdot}}}\hspace{0.4mm}}
\newcommand{\tplus}{\mathrel{\oplus}}

\DeclareMathOperator{\Hilb}{Hilb}
\DeclareMathOperator{\inn}{in}

\DeclareMathOperator{\sgn}{sgn}
\DeclareMathOperator{\id}{id}
\DeclareMathOperator{\lcm}{lcm}
\DeclareMathOperator{\trop}{trop}

\usepackage[bookmarksnumbered=true]{hyperref}
\hypersetup{
  colorlinks = true,
  linkcolor = blue,
  anchorcolor = blue,
  citecolor = blue,
  filecolor = blue,
  urlcolor = blue
}

\begin{document}

\title[The spine of the $T$-graph]{The spine of the $T$-graph of the Hilbert scheme of points in the plane}

\author{Diane Maclagan}
\author{Rob Silversmith}
\address{Mathematics Institute\\ University of Warwick\\ Coventry, CV4 7AL \\ United Kingdom}
\email{D.Maclagan@warwick.ac.uk}
\email{Rob.Silversmith@warwick.ac.uk}

\begin{abstract}
  The torus $T$ of projective space also acts on the Hilbert scheme of
  subschemes of projective space. The $T$-graph of the Hilbert scheme
  has vertices the fixed points of this action, and edges connecting
  pairs of fixed points in the closure of a one-dimensional orbit.  In
  general this graph depends on the underlying field. We construct a
  subgraph, which we call the spine, of the $T$-graph of
  $\Hilb^m(\mathbb A^2)$ that is independent of the choice of infinite
  field.  For certain edges in the spine we also give a description of
  the tropical ideal, in the sense of tropical scheme theory, of a
  general ideal in the edge. This gives a more refined understanding
  of these edges, and of the tropical stratification of the Hilbert
  scheme.
\end{abstract}  

\maketitle

\section{Introduction}

The torus $T \cong (K^*)^n$ of $\mathbb P^n$ acts on the Hilbert
scheme $\Hilb_P(\mathbb P^n)$ of subschemes of $\mathbb P^n$.  There
are finitely many fixed points of this action, but infinitely many
one-dimensional orbits.  The $T$-graph of the Hilbert scheme has
vertices the fixed points of the $T$-action. 
There is an edge between two vertices if there is a one-dimensional
$T$-orbit containing a $K$-rational point whose closure contains these
two vertices.
The $T$-graph provides a combinatorial skeleton of the Hilbert scheme;
for example, the proof that $\Hilb_P(\mathbb P^n)$ is connected given
by Peeva and Stillman \cite{PeevaStillman} proceeds by showing the
Borel-fixed subgraph of this graph is connected (the original proof by
Hartshorne \cite{Hartshorne} has some moves which, while
combinatorial, leave this graph). The $T$-graph of the Hilbert scheme
was first systematically studied by Altmann-Sturmfels
\cite{AltmannSturmfels}, who gave an algorithm to compute it using
Gr\"obner bases, and was studied combinatorially by Hering-Maclagan
\cite{HeringMaclagan}. More generally, $T$-graphs arise in GKM theory
\cite{GKM}, where they are used to give a presentation of the
equivariant cohomology ring of a variety with $T$-action.

  The $T$-graph of the Hilbert scheme $\Hilb^4(\mathbb A^2)$ of $4$
points in $\mathbb A^2$ is shown on the left of
Figure~\ref{f:firsteg}.
Note that a single edge
may correspond to multiple one-dimensional $T$-orbits, or even to a
positive-dimensional family of them.

\begin{figure}
  \centering
  \begin{tikzpicture}[xscale=1.5,yscale=0.75]
    \draw (0,2) node {\begin{tikzpicture}[scale=0.3]
        \foreach \x in {0,1,...,3} {
          \draw (\x,0)--(\x+1,0)--(\x+1,1)--(\x,1)--cycle;
        };
      \end{tikzpicture}};
    \draw (-2,1) node {\begin{tikzpicture}[scale=0.3]
        \foreach \x in {0,1,...,2} {
          \draw (\x,0)--(\x+1,0)--(\x+1,1)--(\x,1)--cycle;
        };
        \foreach \x in {0} {
          \draw (\x,1)--(\x+1,1)--(\x+1,2)--(\x,2)--cycle;
        };
      \end{tikzpicture}};
    \draw (-1,0) node {\begin{tikzpicture}[scale=0.3]
        \foreach \x in {0,1} {
          \draw (\x,0)--(\x+1,0)--(\x+1,1)--(\x,1)--cycle;
        };
        \foreach \x in {0,1} {
          \draw (\x,1)--(\x+1,1)--(\x+1,2)--(\x,2)--cycle;
        };
      \end{tikzpicture}};
    \draw (-2,-1) node {\begin{tikzpicture}[scale=0.3]
        \foreach \x in {0,1} {
          \draw (\x,0)--(\x+1,0)--(\x+1,1)--(\x,1)--cycle;
        };
        \foreach \x in {0} {
          \draw (\x,1)--(\x+1,1)--(\x+1,2)--(\x,2)--cycle;
        };
        \foreach \x in {0} {
          \draw (\x,2)--(\x+1,2)--(\x+1,3)--(\x,3)--cycle;
        };
      \end{tikzpicture}};
    \draw (0,-2) node {\begin{tikzpicture}[scale=0.3]
        \foreach \y in {0,1,...,3} {
          \draw (0,\y)--(1,\y)--(1,\y+1)--(0,\y+1)--cycle;
        };
      \end{tikzpicture}};
    \draw[shorten >=0.5cm,shorten <=1cm,very thick] (0,2)--(-2,1);
    \draw[shorten >=0.5cm,shorten <=0.5cm,very thick] (0,2)--(-1,0);
    \draw[shorten >=1cm,shorten <=0.5cm,very thick] (0,2)--(0,-2);
    \draw[shorten >=0.5cm,shorten <=0.75cm,very thick] (-2,1)--(-1,0);
    \draw[shorten >=.5cm,shorten <=0.4cm,very thick] (-2,1)--(-2,-1);
    \draw[shorten >=0.5cm,shorten <=0.5cm,very thick] (-1,0)--(-2,-1);
    \draw[shorten >=0.9cm,shorten <=0.5cm,very thick] (-1,0)--(0,-2);
    \draw[shorten >=0.5cm,shorten <=0.5cm,very thick] (-2,-1)--(0,-2);
    \foreach \th in {0.1,0.15,...,0.6} {
          \draw[lightgray] plot[smooth,tension=1.2] coordinates {(-2+\th/2,-.35-\th/2)
            (-2+\th,0) (-2+\th/1.4,.5)};
        };
      \end{tikzpicture}
      \quad\quad\quad
      \begin{tikzpicture}[xscale=1.5,yscale=0.75]
    \draw (0,2) node {\begin{tikzpicture}[scale=0.3]
        \foreach \x in {0,1,...,3} {
          \draw (\x,0)--(\x+1,0)--(\x+1,1)--(\x,1)--cycle;
        };
      \end{tikzpicture}};
    \draw (-2,1) node {\begin{tikzpicture}[scale=0.3]
        \foreach \x in {0,1,...,2} {
          \draw (\x,0)--(\x+1,0)--(\x+1,1)--(\x,1)--cycle;
        };
        \foreach \x in {0} {
          \draw (\x,1)--(\x+1,1)--(\x+1,2)--(\x,2)--cycle;
        };
      \end{tikzpicture}};
    \draw (-1,0) node {\begin{tikzpicture}[scale=0.3]
        \foreach \x in {0,1} {
          \draw (\x,0)--(\x+1,0)--(\x+1,1)--(\x,1)--cycle;
        };
        \foreach \x in {0,1} {
          \draw (\x,1)--(\x+1,1)--(\x+1,2)--(\x,2)--cycle;
        };
      \end{tikzpicture}};
    \draw (-2,-1) node {\begin{tikzpicture}[scale=0.3]
        \foreach \x in {0,1} {
          \draw (\x,0)--(\x+1,0)--(\x+1,1)--(\x,1)--cycle;
        };
        \foreach \x in {0} {
          \draw (\x,1)--(\x+1,1)--(\x+1,2)--(\x,2)--cycle;
        };
        \foreach \x in {0} {
          \draw (\x,2)--(\x+1,2)--(\x+1,3)--(\x,3)--cycle;
        };
      \end{tikzpicture}};
    \draw (0,-2) node {\begin{tikzpicture}[scale=0.3]
        \foreach \y in {0,1,...,3} {
          \draw (0,\y)--(1,\y)--(1,\y+1)--(0,\y+1)--cycle;
        };
      \end{tikzpicture}};
    \draw[shorten >=0.5cm,shorten <=1cm,very thick] (0,2)--(-2,1);
    \draw[shorten >=0.5cm,shorten <=0.5cm,very thick] (0,2)--(-1,0);
    \draw[shorten >=1cm,shorten <=0.5cm,very thick] (0,2)--(0,-2);
    \draw[shorten >=.5cm,shorten <=0.4cm,very thick] (-2,1)--(-2,-1);
    \draw[shorten >=0.9cm,shorten <=0.5cm,very thick] (-1,0)--(0,-2);
    \draw[shorten >=0.5cm,shorten <=0.5cm,very thick] (-2,-1)--(0,-2);
  \end{tikzpicture}
  \caption{The $T$-graph and its spine when $N=4$. The gray curves
    illustrate the fact that the edge from $\langle{x^3,xy,y^2}\rangle$ to
    $\langle{x^2,xy,y^3}\rangle$ corresponds to a one-dimensional family of
    $T$-orbits; taking a limit, the $T$-orbits degenerate into the
    union of two orbits, corresponding to the edges from
    $\langle{x^3,xy,y^2}\rangle$ to $\langle{x^2,y^2}\rangle$ and from $\langle{x^2,y^2}\rangle$ to
    $\langle{x^2,xy,y^3}\rangle$.}
  \label{f:firsteg}
\end{figure}
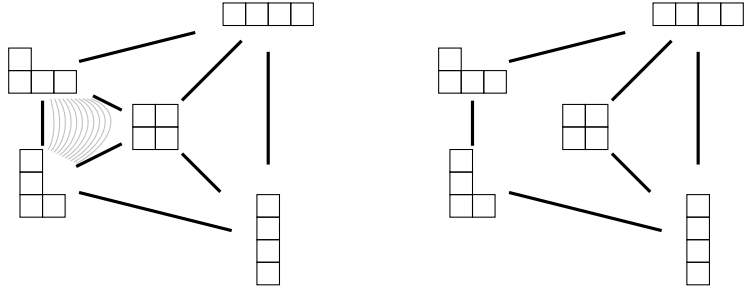

An additional complexity is given by the fact that the graph depends
on the underlying field; the $T$-graph of $\Hilb^{10}(\mathbb A^2)$
differs for $K=\mathbb Q$ and $K=\mathbb R$; see
\cite{HeringMaclagan}*{Example 2.11} and
\cite{SilversmithTropicalIdeal}*{Theorem 5.11}.

The first result of this paper is the construction of a subgraph of
the $T$-graph of the Hilbert scheme $\Hilb^N(\mathbb A^2)$ that does
not depend on the underlying field $K$, provided $K$ is infinite.

A $K$-rational point of $\Hilb^N(\mathbb A^2)$ is a subscheme of
$\mathbb A^2$ of length $N$, given by an ideal $I \subseteq S:=K[x,y]$
with $\dim_K S/I = N$. Such an ideal $I$ is a fixed point of the
$T$-action if and only if it is a monomial ideal; these ideals are in
bijection with Young diagrams with $N$ boxes, with boxes corresponding
to monomials not in $I$.  A non-monomial ideal $I$ lies on a
one-dimensional orbit if and only if $I$ is homogeneous with respect
to a grading by $\deg(x)=a$ and $\deg(y)=b$; the subscheme of $\mathbb
A^2$ defined by $I$ is stabilized by the subtorus $\{ ((t^a,t^b) : t
\in K^* \} \subseteq T$.  There are two $T$-fixed points in the
closure of the orbit, so $I$ corresponds to an edge of the $T$-graph,
if and only if $ab>0$.

In this latter case, denote by $h: \mathbb Z_{\geq 0} \rightarrow \mathbb Z_{\geq 0}$
the Hilbert function of $I$ with respect to this grading: $h(d) :=
\dim_K (S/I)_d$.
Then $I$ lies on the multigraded Hilbert scheme
$\Hilb^h_S\subseteq\Hilb^N(\mathbb{A}^2)$ parametrizing homogeneous
ideals in $S$ with Hilbert function $h$ \cite{HaimanSturmfels}.  This
multigraded Hilbert scheme is a $T$-invariant closed subscheme of
$\Hilb^N(\mathbb{A}^2)$, is smooth and irreducible
\cite{Evain}*{Theorem 1} \cite{MaclaganSmith}*{Theorem 1.1}, and has
two distinguished $T$-fixed points: the ``lex-most'' and ``lex-least''
monomial ideals. See Section~\ref{s:spine} and in particular Figure
\ref{fig:GradedDominanceOrder} for more details.

\begin{definition} \label{d:spine}
  The {\em spine} $G_N^*$ of the $T$-graph of $\Hilb^N(\mathbb A^2)$
  is the graph with vertices the $T$-fixed points of
  $\Hilb^N(\mathbb A^2)$, and an edge between two monomial ideals if
  they are the lex-most and lex-least ideals of $\Hilb_S^h$ with
  respect to some grading and Hilbert function.
\end{definition}

Studying the spine was suggested in Remark 4.7 of
\cite{HeringMaclagan}.  Every one-dimensional $T$-orbit corresponding
to an edge of the $T$-graph is in the closure of the set of $T$-orbits
corresponding to edges in the spine.  The spine for $\Hilb^4(\mathbb
A^2)$ is shown on the right in Figure~\ref{f:firsteg}.  Let $G_{N}(K)$
denote the $T$-graph of $\Hilb^{N}(\mathbb A^2)$ over a field $K$.
Our first theorem is the following.

\begin{theorem} \label{t:spineedge} For any infinite field $K$,
  $G_N^*$ is a subgraph of $G_{N}(K)$; that is, if $M^-$ and $M^+$ are
  the lex-least and lex-most monomial ideals with respect to some
  grading and Hilbert function, then there exists an ideal
  $I\subseteq K[x,y]$, homogeneous with respect to this grading and
  Hilbert function, such that the closure of the $T$-orbit of $I$
  contains $M^-$ and $M^+.$
\end{theorem}

Our second result, Theorem \ref{t:tropicalideal} below, refines
Theorem \ref{t:spineedge} for some edges by describing matroidal aspects of
$\Hilb_S^h$ coming from tropical scheme theory. We now describe what
we mean by this; for precise definitions, see Section 
\ref{sec:Tropicalization}. 

The {\em tropicalization} $\trop(I)$ of an ideal $I \subseteq K[x,y]$
is the ideal in the semiring of tropical polynomials obtained by
tropicalizing every polynomial in the ideal.  This is an example of a
tropical ideal in the sense of tropical scheme theory
\cites{Giansiracusa2, TropicalIdeals, MaclaganRinconValuations, Balancing}.
When $I$ is homogeneous, each degree-$d$ part of $\trop(I)$  determines a {\em matroid} $\mathcal{M}(I_d)$ on the set $\QQ{d}$ of degree-$d$ monomials.

This construction induces a \emph{tropical stratification} of
$\Hilb_S^h$; two ideals are in the same stratum if and only if their
tropicalizations coincide.  This can be thought of as a generalization
of the matroid stratification of the Grassmannian \cite{GGMS}.  Very
little is known about the tropical stratification;  see
\cites{SilversmithTropicalIdeal,FinkGiansiracusaGiansiracusa}.

When $n=2$, and the grading is the standard one $\deg(x)=\deg(y)=1$, the Hilbert scheme $\Hilb_S^h$ is irreducible \cite{Evain,MaclaganSmith}, and
hence has a unique open (largest) stratum. Our second main theorem,
Theorem \ref{t:tropicalideal} below, describes this stratum; in other
words, it describes the tropicalization of a general ideal
$I$ in $\Hilb_S^h.$

\begin{theorem} \label{t:tropicalideal}
  Let $S = K[x,y]$ be graded by $\deg(x)=\deg(y)=1$.
  For any $d \geq 0$, the degree-$d$ matroid $\mathcal{M}(I_d)$ of a
  general ideal $I$ in $\Hilb_S^h$ is the uniform matroid
  $U_{h(d),d+1}.$ Furthermore, $I$ can be taken to be a $K$-rational
  point of $\Hilb_S^h$, provided $K$ is infinite.
\end{theorem}

Theorem \ref{t:tropicalideal} fails in the nonstandard grading;
see Section~\ref{sec:Nonstandard}.

There are comparatively few explicit examples of tropical ideals; see
\cites{Zajaczkowska,AndersonRincon}.  One important aspect of
Theorem~\ref{t:tropicalideal} is thus that it provides a large class
of new examples for which all matroids are understood.

Theorem \ref{t:tropicalideal} refines Theorem \ref{t:spineedge}
as follows. For a fixed grading and Hilbert function, the ideals
$I\subseteq K[x,y]$ whose orbit contains $M^-$ and $M^+$ comprise an
open set $U_1\subseteq\Hilb_S^h$, which is nonempty by Theorem
\ref{t:spineedge}. Meanwhile, the ideals $I\subseteq K[x,y]$ such that
the conclusion of Theorem \ref{t:tropicalideal} holds for all $d\ge0$
also comprise a nonempty open set $U_2\subseteq\Hilb_S^h$, and we have
the containment $U_2\subseteq U_1$; see Remark
\ref{rem:InitialRowReduction}.

The structure of this paper is as follows.  In Section~\ref{s:spine}
we give more precise definitions of the main objects of study, and
prove Theorem~\ref{t:spineedge}.  Theorem~\ref{t:tropicalideal} 
is proved in
Section~\ref{s:tropicalideal}.

\noindent {\bf Acknowledgements.}  Maclagan was partially supported by
EPSRC grant EP/R02300X/1. Silversmith was supported by NSF DMS-1645877
and by a Zelevinsky postdoctoral fellowship at Northeastern
University.  Some background calculations were done using Macaulay2
\cite{M2}.

\section{The spine of the \texorpdfstring{$T$}{T}-graph} \label{s:spine}

In this section we recall previous work on the $T$-graph, and prove
 Theorem~\ref{t:spineedge}.

Let $K$ be an infinite field. Recall that a $K$-rational point of $\Hilb^N(\mathbb A^2)$ is given by
an ideal $I \subseteq S:=K[x,y]$ with $\dim_K(S/I) = N$. The
$T \cong (K^*)^2$ action on $\mathbb A^2$ induces a $T$-action on
$\Hilb^N(\mathbb A^2)$. Such an ideal is a fixed point of the
$T \cong (K^*)^2$ action on $\Hilb^N(\mathbb A^2)$ if and only if it
is a monomial ideal, and lies on a one-dimensional $T$-orbit if and
only if it is homogeneous with respect to a $\mathbb Z$-grading
by
$\deg(x)=a$ and $\deg(y)=b$.
The closure of a
one-dimensional $T$-orbit has either one or two $T$-fixed points; if
there are two $T$-fixed points in the closure we have $ab>0$.

\begin{notation}
  We set $S = K[x,y]$.  We grade $S$ by
  $\deg(x)=a$, $\deg(y)=b$, for positive integers $a,b$, and denote
  this as an $(a,b)$-grading. From now on we restrict to $a,b>0$ and
  $\gcd(a,b)=1$; this makes no material difference, and will simplify
  our notation.
  \end{notation}

Let $I\in\Hilb^N(\mathbb{A}^2)$ be an ideal that is homogeneous with
respect to the grading by $(a,b)\in\mathbb{Z}_{>0}^2$. The Hilbert
function $h:\mathbb{Z}_{\ge0}\to\mathbb{Z}_{\ge0}$ of $I$ is defined
by $h(d)=\dim_K(S/I)_d.$ Note that $\sum_{d\ge0}h(d)=N.$ The point
$I\in\Hilb^N(\mathbb{A}^2)$ is contained in the closed subscheme
$\Hilb_S^{h}$ parametrizing ideals in $S$ that are homogeneous with
respect to the $(a,b)$-grading and have Hilbert function $h$. This
subscheme is a \emph{multigraded Hilbert scheme} in the sense of
\cite{HaimanSturmfels}. Furthermore, for any grading $(a,b)\in
\mathbb{Z}_{>0}^2$ and for any Hilbert function
$h:\mathbb{Z}_{\ge0}\to\mathbb{Z}_{\ge0}$ with $\sum_{d\ge0}h(d)=N,$
if the scheme $\Hilb_S^{h}$ is nonempty, it is a smooth irreducible $T$-invariant
subvariety of $\Hilb^N(\mathbb{A}^2)$ \cite{Evain,MaclaganSmith}. The
union of the subvarieties $\Hilb_S^h$, as $(a,b)$ and $h$ vary, is
precisely the set of ideals corresponding to vertices and edges of the
$T$-graph, and any intersection of two different $\Hilb_S^h$ is either
empty, or consists of a single point corresponding to a monomial
ideal.

Each multigraded Hilbert scheme $\Hilb_S^h$ inherits a $T$-action from
$\Hilb^N(\mathbb{A}^2).$ We may define the $T$-graph $G_h(K)$ of
$\Hilb_S^h$ in exact analogy with that of $\Hilb^N(\mathbb{A}^2)$: the
vertices of $G_h(K)$ are zero-dimensional $T$-orbits, which are
monomial ideals whose Hilbert function with respect to $(a,b)$ is $h$,
and two vertices are connected by an edge if there is a
one-dimensional $T$-orbit in $\Hilb_S^h$ containing a $K$-rational
point whose closure contains those vertices. Note that $G_h(K)$ is
naturally a subgraph of $G_N(K)$. Moreover, we have the following
decomposition:
\begin{proposition}[\cite{HeringMaclagan}, Corollary 2.6]\label{cor:HMCor}
  The $T$-graph $G_N(K)$ is the union of the subgraphs $G_h(K)$ as
  $(a,b)$ and $h$ vary, and these subgraphs have disjoint edge sets.
\end{proposition}
In light of Proposition \ref{cor:HMCor}, in order to determine $G_N(K)$
it is sufficient to study the graded Hilbert schemes $\Hilb_S^h$
separately. Thus from now on, fix a grading
$(a,b)\in\mathbb{Z}_{>0}^2$ with $\gcd(a,b)=1$, and a Hilbert function
$h:\mathbb{Z}_{\ge0}\to\mathbb{Z}_{\ge0}$ with $\sum_{d\ge0}h(d)<\infty.$

Note that the 1-parameter subtorus
$T_{a,b}:=\{(t^a,t^b)\thickspace|\thickspace t\in
K^*\}\subseteq T$ acts trivially on $\Hilb_S^h$, so we need only
consider the action of the one-dimensional torus $T/T_{a,b}$. Since
$\Hilb_S^h$ is smooth and projective, the Bia{\l}ynicki-Birula decomposition of
$\Hilb_S^h$ with respect to $T/T_{a,b}$ decomposes $\Hilb_S^h$ as a
union of affine spaces, each consisting of points whose limit under
the subtorus action is a given fixed point. We describe these affine
spaces algebraically as follows. Set $\prec$ to be the lexicographic
order with $x \prec y$. A $T$-fixed point corresponds to a monomial ideal
$M$. The Bia{\l}ynicki-Birula cell associated to $M$ is
$$C_{\prec}(M) = \{ I \in \Hilb_S^h\thickspace|\thickspace
\inn_{\prec} (I)= M \},$$ where $\inn_{\prec}(I)$ is the initial ideal
in the sense of Gr\"obner bases; see \cite{CLO}.  An explicit
parameterization of $C_{\prec}(M)$ was given by Evain~\cite{Evain}; we
recall this in Section~\ref{sec:ArrowsAndPaths}.  For two monomial
ideals $M, M' \in \Hilb_S^h$, the {\em edge-scheme} between $M$ and
$M'$ is the scheme-theoretic intersection
$$E(M,M') = C_{\prec}(M) \cap C_{\prec^{opp}}(M'),$$ where
$\prec^{opp}$ is the lexicographic order with $x \succ y$.  This was
first studied computationally by Altmann and
Sturmfels~\cite{AltmannSturmfels}.  There is an edge between $M$ and
$M'$ in the $T$-graph if and only if one of $E(M,M')$ and $E(M',M)$
has a $K$-rational point.

The vertices of $G_N(K)$ are purely combinatorial: colength-$N$
monomial ideals in $S = K[x,y]$ correspond to partitions of $N$, and
requiring
that the
Hilbert function with respect to $(a,b)$ is $h$ is a combinatorial
condition on partitions of $N =\sum_{d\ge0}h(d)$.
The edges, however, depend on the field $K$, as the following examples
show.

\begin{example}
  There is an edge between the two monomial ideals
  $M= \langle x^5,y^2 \rangle$ and $M'=\langle x^2,y^5 \rangle$ when
  viewed as ideals in $\mathbb R[x,y]$, but not when viewed as ideals
  in $\mathbb Q[x,y]$, so the $T$-graph of $\Hilb^{10}(\mathbb A^2)$
  differs over $\mathbb R$ and ${\mathbb Q}$. This is the case because
  every ideal in $E(M,M')$ has the form
  $I = \langle y^2+\alpha xy+\beta x^2, x^5 \rangle$, with
  $\alpha^4-3\alpha^2\beta+\beta^2=0$, by
  \cite{HeringMaclagan}*{Example 2.11}.  This is the union of two
  one-dimensional $T$-orbits, which have $\mathbb R$-rational points,
  but no $\mathbb Q$-rational points. Note that the edge scheme
  $E(M,M')$ is a subscheme of $\Hilb_S^h$, where the grading is
  $(1,1),$ and $h=(1,2,2,2,2,1,0,0,\ldots).$

  This example is generalized in
  \cite{SilversmithTropicalIdeal}*{Theorem 5.11}, which shows that if
  $m>k>0,$ and we define $M=\langle x^m,y^k \rangle$ and
  $M'=\langle x^k, y^m \rangle$ in $\Hilb^{mk}(\mathbb A^2)$, then the
  edge-scheme $E(M,M')$ is one dimensional and reducible over
  $\mathbb C$, with the number of irreducible components equal to the
  number of binary necklaces with $k$ black and $m-k$ white beads.
  The proof actually shows that these edges have $K$-rational points
  whenever there exists $c\in K$ such that $x^m+cy^m$ has a degree-$k$
  factor with coefficients in $K$.
\end{example}

By contrast, the definition of the spine of the $T$-graph, given in Definition~\ref{d:spine},  is purely combinatorial.

\begin{definition} \label{d:dominanceorder}
  Let $M,M'$ be monomial ideals in $\Hilb_S^h$. We define
  $M \preceq M'$ if for each degree $d$ there is a degree-preserving
  bijection $f$ from the monomials in $M$ to the monomials in $M'$
  with $m \succeq f(m)$ for all monomials $m \in M$, where
  $\prec$ is the lexicographic order with $x \prec y$. This defines a
  partial ordering on the monomial ideals in $\Hilb_S^h$.
\end{definition}
\begin{remark}
  The partial order of Definition~\ref{d:dominanceorder} may be
  regarded as a graded version of the dominance order for partitions.
  Recall that the monomial ideals $M,M'$ correspond to partitions, or
  alternatively, to Young diagrams.  Under this correspondence, $M\preceq M'$
  if and only if $M'$ can be obtained from $M$ by moving boxes of the
  Young diagram up and to the left, along lines of slope $-a/b.$ See
  Figure \ref{fig:GradedDominanceOrder}. Compare this to the usual
  dominance order for partitions, which is identical after removing
  the slope restriction.
\end{remark}

 A necessary, but not sufficient, condition for
  $E(M,M')$ to be nonempty is that $M\preceq M'$ with respect to this partial 
 order; this is a straightforward special case of \cite[Thm.
   1.3]{HeringMaclagan}. In particular, if $M\ne M',$ then at most one of $E(M,M')$ and $E(M',M)$ is nonempty. We may therefore regard $G_N(K)$ as a directed graph, with an edge from $M$ to $M'$ if $E(M,M')$ is nonempty; the necessary condition above implies that the resulting directed graph is acyclic.

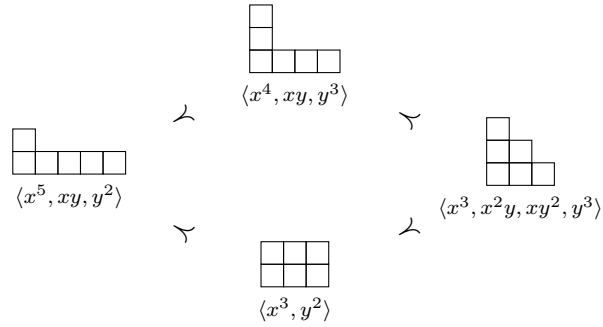
\begin{figure}
  \centering
  \begin{tikzpicture}[scale=1.5]
    \draw (0,0) node {\begin{tikzpicture}[scale=0.3]
        \foreach \x in {0,1,...,4} {
          \draw (\x,0)--(\x+1,0)--(\x+1,1)--(\x,1)--cycle;
        };
        \foreach \x in {0} {
          \draw (\x,1)--(\x+1,1)--(\x+1,2)--(\x,2)--cycle;
        };
        \draw (2.5,-1) node {\tiny $\langle{x^5,xy,y^2}\rangle$};
      \end{tikzpicture}};
    \draw (2,1) node {\begin{tikzpicture}[scale=0.3]
        \foreach \x in {0,1,...,3} {
          \draw (\x,0)--(\x+1,0)--(\x+1,1)--(\x,1)--cycle;
        };
        \foreach \x in {0} {
          \draw (\x,1)--(\x+1,1)--(\x+1,2)--(\x,2)--cycle;
        };
        \foreach \x in {0} {
          \draw (\x,2)--(\x+1,2)--(\x+1,3)--(\x,3)--cycle;
        };
        \draw (2,-1) node {\tiny $\langle{x^4,xy,y^3}\rangle$};
      \end{tikzpicture}};
    \draw (2,-1) node {\begin{tikzpicture}[scale=0.3]
        \foreach \x in {0,1,...,2} {
          \draw (\x,0)--(\x+1,0)--(\x+1,1)--(\x,1)--cycle;
        };
        \foreach \x in {0,1,2} {
          \draw (\x,1)--(\x+1,1)--(\x+1,2)--(\x,2)--cycle;
        };
        \draw (1.5,-1) node {\tiny $\langle{x^3,y^2} \rangle$};
      \end{tikzpicture}};
    \draw (4,0) node {\begin{tikzpicture}[scale=0.3]
        \foreach \x in {0,1,2} {
          \draw (\x,0)--(\x+1,0)--(\x+1,1)--(\x,1)--cycle;
        };
        \foreach \x in {0,1} {
          \draw (\x,1)--(\x+1,1)--(\x+1,2)--(\x,2)--cycle;
        };
        \foreach \x in {0} {
          \draw (\x,2)--(\x+1,2)--(\x+1,3)--(\x,3)--cycle;
        };
        \draw (1.5,-1) node {\tiny $\langle{x^3,x^2y,xy^2,y^3}\rangle$};
      \end{tikzpicture}};
    \draw (1,.5) node[rotate=30] {$\prec$};
    \draw (1,-.5) node[rotate=-30] {$\prec$};
    \draw (3,.5) node[rotate=-30] {$\prec$};
    \draw (3,-.5) node[rotate=30] {$\prec$};
  \end{tikzpicture}  
  \caption{The poset of monomial ideals in $\Hilb^h_S$, where
    $h=(1,1,2,1,1,0,0,\ldots)$ with respect to the $(1,2)$-grading.
    The leftmost ideal is the lex-least, and the rightmost
    ideal is the lex-most. Note that $M_1\prec M_2$ if and only if
    $M_2$ can be obtained from $M_1$ by moving boxes of the Young
    diagram up-and-left along lines of slope $-1/2$.}
  \label{fig:GradedDominanceOrder}
\end{figure}

It was first noted by Evain \cite{Evain}*{Theorem 19} that this poset
has a unique maximal element, which we denote by $M^+$, and a
unique minimal element, which we denote by $M^-$; see also
\cite{MaclaganSmith}*{Proposition 3.12}. We call $M^+$ the lex-most
ideal with Hilbert function $h$, and $M^-$ the lex-least such ideal.

As defined in Definition~\ref{d:spine}, the spine $G^*_N$ of the
$T$-graph of $\Hilb^N(\mathbb A^2)$ is the graph with vertices
monomial ideals $M$ in $S$ with $\dim_K(S/M) =N$, and an edge joining
two ideals $M,M'$ if $M=M^-$, and $M'=M^+$ for some grading $(a,b)$
and Hilbert function. Figure \ref{fig:spine6} shows $G_6(\mathbb C)$
and $G_6^*$.

\begin{figure}
  \centering
  \includegraphics[height=3in]{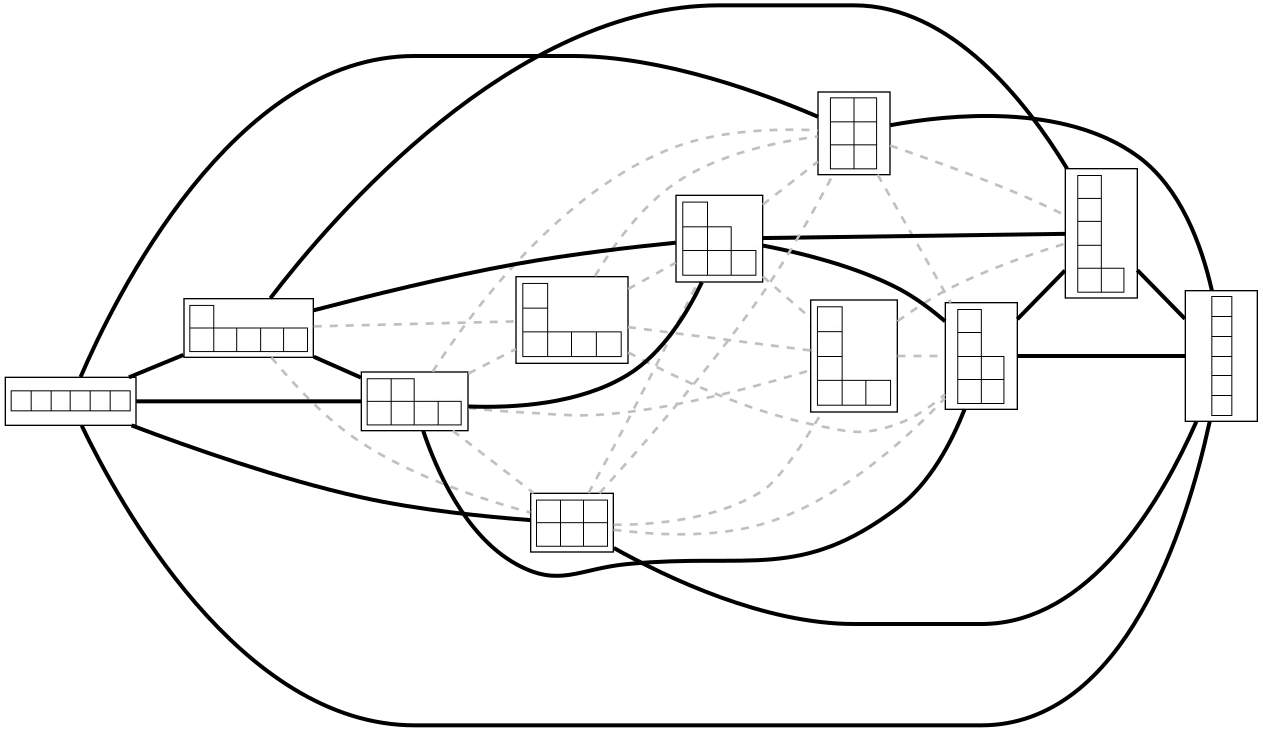}
  \caption{The $T$-graph $G_6(\mathbb C)$, where the edges of $G_6^*$ are
    solid.}
  \label{fig:spine6}
\end{figure}

\begin{theorem} \label{t:algebraicallyclosedcase} Let $K$ be an
  infinite field.  If there is an edge connecting two monomial ideals
  $M,M'$ in $G^*_N$, then there is an edge in the $T$-graph $G_N(K)$
  connecting $M,M'$.
\end{theorem}

\begin{proof}
    Suppose $M$ and $M'$ are connected by an edge in $G^*_N.$ By
    definition, there exists a grading $(a,b)$ with respect to which
    the Hilbert functions of $M$ and $M'$ agree, and after possibly
    renaming $M$ and $M'$, we have $M=M^+$ and $M'=M^-.$

    By \cite[Thm. 1]{Evain}, \cite[Thm. 1.1]{MaclaganSmith},
    $\Hilb_S^h$ is smooth, projective, and irreducible.  Since $M^+$
    is a source, and $M^-$ is a sink, of the $T/T_{a,b}$ action, the
    Bia\l{}ynicki-Birula cells $C_{\prec^{opp}}(M^+)$ and
    $C_{\prec}(M^-)$ are Zariski open and isomorphic to affine spaces
    (\cite{BialynickiBirula},\cite[Thm. 11]{Evain}). Note that this follows from \cite{BialynickiBirula} only when the field $K$ is algebraically closed, but this assumption is unnecessary --- for a discussion see \cite[\S 3]{Brosnan}. Also, while \cite{Evain} assumes $K$ is algebraically closed, this is never used in the proofs. Thus $C_{\prec^{opp}}(M^+)\cap
    C_{\prec}(M^-)$ is isomorphic to an open subset of an affine space
    over $K$; since $K$ is infinite, it follows that
    $C_{\prec^{opp}}(M^+)\cap C_{\prec}(M^-)$ contains a $K$-rational
    point.
\end{proof}

\section{The tropical ideal of an edge of the
  spine} \label{s:tropicalideal}

In this section we prove Theorem
\ref{t:tropicalideal}.

\subsection{Tropicalizations of ideals}\label{sec:Tropicalization}

We first recall the concept of tropicalization of ideals, and the
tropical stratification of the Hilbert scheme.

Let $\mathbb B = (\{0,\infty\},\tplus, \ttimes) $ be the Boolean
semiring, with the operations of tropical addition (minimum) and
tropical multiplication (addition).  The tropicalization of $f = \sum
c_{ij} x^iy^j \in K[x,y]$ is
$\trop(f) = \tplus_{c_{ij} \neq 0} x^iy^j \in \mathbb
B[x,y]$.  The tropicalization of an ideal $I \subseteq
K[x,y]$ is $$\trop(I) = \langle \trop(f) : f \in I
\rangle \subseteq \mathbb B[x,y].$$

This is the trivial valuation case of tropicalizing ideals in the
sense of tropical scheme theory
\cites{Giansiracusa2, TropicalIdeals, MaclaganRinconValuations, Balancing}.

Note that a polynomial in the semiring $\mathbb B[x,y]$ can be
identified with its support.  When $I \subseteq S$ is graded, the
polynomials in $\trop(I)$ of degree $d$ of minimal support are the
{\em circuits} of a matroid $\mathcal M(I_d)$ on the ground set
$\QQ{d}$ of degree-$d$ monomials.  We call this the degree-$d$ matroid
of $I$.  See, for example, \cite{Oxley} for more on matroids.

We will primarily focus on the {\em basis} characterization of
matroids.  When $I \subseteq S$ is homogeneous with Hilbert function
$h$, a collection $E$ of $h(d)$ monomials of degree $d$ is a basis for
$\mathcal M(I_d)$ if there is no polynomial in $I$ with support in
$E$.  The matroid $\mathcal M(I_d)$ is {\em uniform} if every
collection of $h(d)$ monomials of degree $d$ is a basis.  In this case
we write $\mathcal M(I_d) = U_{h(d),\qq{d}}$, where $\qq{d} =
|\QQ{d}|$.

The assignment $I\mapsto \trop(I)$ defines a stratification of
$\Hilb_S^h$, called the \emph{matroid stratification} or
\emph{tropical stratification}. A stratum of this stratification
consists of all ideals with a fixed tropicalization. If $\sum_dh(d)<\infty$ (as will always
be true in this paper), then there are finitely many strata, and they
are Zariski-locally closed. In general, there may be countably many
strata; see \cite{SilversmithTropicalIdeal}.

\subsection{Evain's parameterization of the Bia{\l}ynicki-Birula cells}\label{sec:ArrowsAndPaths}

In this section we recall Evain's parameterization of the
Bia{\l}ynicki-Birula cells.
This relies on the
combinatorial decomposition of the tangent space to the Hilbert scheme
at a monomial ideal given by {\em significant arrows}.

\begin{notation}
  Let $\prec$ denote the lexicographic order on monomials in $S = K[x,y]$ with
  $x\prec y$.  We set $r$ to be the Laurent monomial $ x^b/y^a$.
  When $(a,b)=(1,1)$, we have $r = x/y$.
\end{notation}

\begin{definition}\label{d:MiWiArrows}
Let $M \subseteq S$ be a finite-colength monomial ideal.  Write the
minimal generators for $M$ as $m_0 \prec m_1 \prec \dots \prec m_e$,
so $m_0$ is a power of $x$, and $m_e$ is a power of $y$.  For $1 \leq
i \leq e$ set $w_i = \lcm(m_i,m_{i-1})$.

The set $T^+(M)$ is 
$$T^+(M) = \{ (i,\ell) :1 \leq i \leq e, \ell \in \mathbb Z_{\geq 1},
m_ir^{\ell} \in S \setminus M, w_ir^{\ell} \in M \}.$$ Elements of $T^+(M)$
are often drawn as arrows from $m_i$ to $m_ir^{\ell}$, and are called
{\em positive significant arrows}.  This is illustrated in Figure~\ref{f:sigarrowdefn}.

The set $T^-(M)$ of {\em negative significant arrows} has arrows pointing in the other direction:
$$T^-(M) = \{ (i,\ell) :0 \leq i \leq e-1, \ell \in \mathbb Z_{\leq -1},
m_ir^{\ell} \in S \setminus M, w_{i+1}r^{\ell} \in M \}.$$
\end{definition}

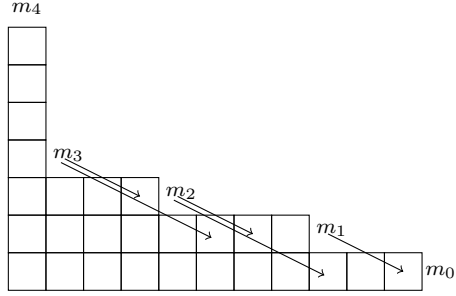
\begin{figure}
  \centering
  \begin{tikzpicture}[scale=0.5]
    \foreach \x in {0,1,...,10} {
      \draw (\x,0)--(\x+1,0)--(\x+1,1)--(\x,1)--cycle;
    };
    \foreach \x in {0,...,7} {
      \draw (\x,1)--(\x+1,1)--(\x+1,2)--(\x,2)--cycle;
    };
    \foreach \x in {0,...,3} {
      \draw (\x,2)--(\x+1,2)--(\x+1,3)--(\x,3)--cycle;
    };
    \foreach \y in {3,...,6} {
      \draw (0,\y)--(1,\y)--(1,\y+1)--(0,\y+1)--cycle;
    };
    \draw (11.5,.5) node {\tiny $m_0$};
    \draw (8.6,1.6) node {\tiny $m_1$};
    \draw (4.6,2.6) node {\tiny $m_2$};
    \draw (1.6,3.6) node {\tiny $m_3$};
    \draw (0.5,7.5) node {\tiny $m_4$};
    \draw[->] (8.5,1.5)--(10.5,0.5);
    \draw[->] (4.5,2.5)--(6.5,1.5);
    \draw[->] (4.4,2.4)--(8.4,0.4);
    \draw[->] (1.5,3.5)--(3.5,2.5);
    \draw[->] (1.4,3.4)--(5.4,1.4);
  \end{tikzpicture}
  \caption{The ideal $M=\langle{x^{11},x^8y,x^4y^2,xy^3,y^7}\rangle$ has
    $T^+(M)=\{(1,1),(2,1),(2,2),(3,1),(3,2)\}$ with respect to the grading $(1,2).$}
  \label{f:sigarrowdefn}
\end{figure}

\begin{remark}
  The use of arrows as a combinatorial basis for the tangent space
  of the Hilbert scheme of points at a monomial ideal was introduced by Haiman in 
  \cite{HaimanQTCatalan}.  In Haiman's formulation there is an equivalence class of arrows; we follow the convention introduced in  \cite{Evain} to choose a particular representative of this class that starts at a minimal generator of the ideal, and use the notation from   \cite{MaclaganSmith}.
\end{remark}

The lex-most and lex-least ideals $M^+$ and $M^-$ defined in Section~\ref{s:spine} can be characterized as the unique ideals with $T^+(M^+)=T^-(M^-)= \emptyset$.  
This was first shown in \cite{Evain}, and generalized in
\cite{MaclaganSmith}.

We next recall the construction of the universal ideal over
$C_{\prec}(M)$.

\begin{definition}\label{d:fk}
  For a monomial $m\in M,$ we define
  \begin{align*}
    j^+(m)&=\max\{i\thickspace|\thickspace\text{$m_i$ divides $m$}\}, \text{ and }\\
    j^-(m)&=\min\{i\thickspace|\thickspace\text{$m_i$ divides $m$}\}.
  \end{align*}
  Note $j^+$ is denoted by $j$ in \cite{HeringMaclagan}.
  We form the polynomial ring
$K[\{c_i^\ell\thickspace|\thickspace(i,\ell)\in
  T^+(M)\}]$ with variables indexed by $T^+(M)$, and 
  recursively define polynomials
  $$f_0,\ldots,f_e\in K[\{c_i^\ell\thickspace|\thickspace(i,\ell)\in
  T^+(M)\}][x,y]$$ by $f_0=m_0$ and
  $$f_i=\frac{m_i}{m_{i-1}}f_{i-1}+\sum_{(i,\ell)\in
    T^+(M)}c_i^\ell\frac{m_ir^\ell}{m_{j^+(w_ir^\ell)}}f_{j^+(w_ir^\ell)}.$$
  Note that the initial (leading)
  term of $f_i$ with respect to
  $\prec$ is $m_i$.
\end{definition}
\begin{theorem}[\cite{Evain}, Theorem 11] \label{t:evainuniversal}
   The set $\{f_0,\ldots,f_e\}$ is a Gr\"obner basis for the universal
   ideal $I$ over $C_{\prec}(M)$. The induced map
   $\A^{\abs{T^+(M)}}\to H_S^h$ is injective with image $C_\prec(M).$
\end{theorem}

We will work directly with the coefficients of $f_i$. To do
so, we will use a combinatorial non-recursive description of these
coefficients given in \cite{HeringMaclagan}, which we now describe.
\begin{definition}[\cite{HeringMaclagan}*{Definition 4.10}]\label{d:Path}
  A \emph{path} from a generator
  $m_i\in M$ is a sequence of positive significant arrows
  $P=((i_1,\ell_1),(i_2,\ell_2),\ldots,(i_d,\ell_d))$, such that:
  \begin{enumerate}[label=(\alph*)]
  \item $i_1\le i,$ and
  \item if $d\ge2,$ then $((i_2,\ell_2),\ldots,(i_d,\ell_d))$ is a
    path from $m_{j^+(w_{i_1}r^{\ell_1})}$.
  \end{enumerate}
  The \emph{length} of $P$ is $l(P)=\ell_1+\cdots+\ell_d.$ We also
  associate to $P$ the monomial $c_P=c_{i_1}^{\ell_1}\cdots
  c_{i_d}^{\ell_d}$ in $K[C_{\prec}(M)] = K[c_i^{\ell} : (i,\ell) \in
    T^+(M)]$.  If the sequence is empty, $P$ is the empty path, which
  has length $0$, and $c_P = 1$.
\end{definition}
\begin{example} \label{e:paths}
  In Figure \ref{f:sigarrowdefn}, the paths from $m_3$ are as follows:
  \begin{center}
    \begin{tabular}{cl}
      Length&Paths\\\hline
      1&\textbf{((3,1))},\quad((2,1)),\quad((1,1))\\
      2&\textbf{((3,2))},\quad((3,1),(2,1)),\quad((3,1),(1,1)),\quad((2,2)),\quad((2,1),(1,1))\\
      3&\textbf{((3,2),(1,1))},\quad((3,1),(2,2)),\quad((3,1),(2,1),(1,1))
    \end{tabular}
  \end{center}
  The boldfaced paths are the direct paths, defined in Definition \ref{Def:DirectPath}.
\end{example}
\begin{theorem}[\cite{HeringMaclagan}, Lemma 4.12]\label{t:CoefficientsArePaths}
  We have the following alternate characterization of
  $f_i$:
  $$f_i=\sum_{\substack{\text{$P$ a path }\\\text{from $m_i$}}}c_Pm_ir^{l(P)}.$$
\end{theorem}
Note that the term in this sum corresponding to the empty path is $m_i$.  
\begin{example}
  \label{e:paths2}
  Continuing Example~\ref{e:paths}, we have $f_0=m_0 =x^{11}$, $f_1 =
  x^8y+c_1^1x^{10}$, $f_2 = x^4y^2 +
  (c_1^1+c_2^1)x^6y+(c_2^1c_1^1+c_2^2) x^8$, and $f_3 =
  xy^3+(c_1^1+c_2^1+\mathbf{c_3^1})x^3y^2+(c_2^1c_1^1+c_2^2+c_3^1c_1^1+c_3^1c_2^1+\mathbf{c_3^2})x^5y+(c_3^1c_2^1c_1^1+c_3^1c_2^2+\mathbf{c_3^2c_1^1})x^7$. The
  boldfaced monomials $c_P$ correspond to direct paths, defined in  Definition \ref{Def:DirectPath}.
\end{example}

  We will focus on one monomial $c_P$ in each term of $f_i$, as follows.

\begin{definition}\label{Def:DirectPath}
  Fix $k \geq 1$. For all $j \leq k$, let $\ell_j$ be the longest length of a
  significant arrow $(j,\ell') \in T^+(M^-)$. We construct a sequence
  $(z_1,z_2,\dots)$ of variables $c_i^{\ell'}$ as follows.  Set
  $m=m_k$, and $l=1$.  If $\ell_k>0$, set $z_1 = c_k^{\ell_k}$,
  $i=j^+(w_kr^{\ell_k})$, and $l=2$.  Otherwise set $i=k-1$.  We now
  iterate.  Given $m,l,i$, if $\ell_i>0$, set $z_l = c_i^{\ell_i}$,
  $i=j^+(w_ir^{\ell_i})$, and $l=l+1$.  Otherwise set $i=i-1$.  This
  procedure stops when $i \leq 0$.

  A path $P$ is called a \emph{direct path from $m_k$} if it is of one
  of the two forms $(z_1,z_2,\ldots,z_s)$, with $s>0$, or
  $(z_1,z_2,\ldots,z_s,c_{i'}^{\ell'})$, with $s\ge0$, where the index $i'$ agrees
  with the index of $z_{s+1}$, and $\ell'<\ell_{i'}$.
\end{definition}

\begin{remark}\label{rem:PropertiesOfDirectPaths}
  Note the following properties:
  \begin{enumerate}
  \item There is at most one direct path from $m_k$ of a given length
    $\ell$. This is because a choice of path is determined by $s$ and
    $\ell'$, and the corresponding path has length
    $\ell=\sum_{i=1}^s \ell_i+\ell'<\sum_{i=1}^{s+1} \ell_i$. We refer
    to this path, when it exists, as $p_{k,\ell}$.
  \item For a fixed $k,$ and a fixed positive significant arrow
    $c_i^{\ell'},$ there is at most one $\ell$ such that $c_i^{\ell'}$
    is the \emph{last} step in a direct path $p_{k,\ell}$, in the
    sense that for any other positive significant arrow
    $c_{i'}^{\ell''}$ in $p_{k,\ell},$ we have $i'>i.$
  \item If $P$ is a direct path from $m_k$ of length $\ell>\ell_k$,
    then the path $P'$ obtained by deleting the first step of $P$ is a
    direct path of length $\ell-\ell_k$ from $m_{j^+(w_kr^{\ell_k})}.$
  \end{enumerate}
\end{remark}

When $S$ has the standard grading, we next show that direct paths of
all possible lengths exist from certain monomials $m_k$.  This uses
the following properties of the lex-most and lex-least ideals.

\begin{remark} 
  In the standard grading $\deg(x)=\deg(y)=1,$ the lex-most ideal
  $M^+$ is the {\em lexicographic ideal}, also known as the {\em
    lexsegment ideal}, with respect to the order $x \succ y$; see
  \cite{BrunsHerzog}*{Chapter 4}.  This is the monomial ideal whose
  degree $d$ part is the span of the $(d+1)-h(d)$ largest monomials in
  lexicographic order.  The lex-least ideal $M^-$ is the lexicographic
  ideal for the opposite order of the variables $x \prec y$.  A
  monomial ideal is lex-least with respect to the standard grading if
  and only if the rows of its Young diagram are \emph{strictly}
  decreasing in length, and similarly is lex-most if and only if the
  columns of its Young diagram are strictly decreasing in length.
  This means that for $M^-$ we have $m_k = x^iy^k$ for some $i$, so
  $m_kr^{\ell}\in S$ for $0 \leq \ell \leq k$.  We also have by
  symmetry that if $M^-$ and $M^+$ are the lex-least and lex-most
  monomial ideals with a given Hilbert function $h$ respectively, then
  the Young diagrams of $M^-$ and $M^+$ are transposes of each other.
\end{remark}

Another standard-graded fact about $M^-$ that we need, which is not
true for nonstandard gradings, is that $w_k/m_{k-1} = y$ for all
$k \geq 1$.

\begin{proposition}  \label{p:coeffsoffsubk}
Fix the standard $(1,1)$-grading for $S$.  Fix $k \ge 0$ with $m_kr
\in S \setminus M^-$.  If for some $0 < \ell \leq k$ we have that
$m_kr^{\ell} \in S \setminus  M^-$, then there is a direct path of length
$\ell$ from $m_k$.
\end{proposition}

\begin{proof}
  The proof is by induction on $k$.  When $k=0$, there is no such
  $\ell$, so the claim holds.  Now assume that the claim is true for
  all $k'<k$.  Let $\ell'$ be maximal such that $(k,\ell') \in
  T^+(M^-)$.  If $\ell' \geq \ell$, then we claim that $(k,\ell) \in
  T^+(M^-)$.  This follows from the fact that $w_kr^{\ell'} \in M^-$,
  so since $w_kr^{\ell'} \preceq w_k r^{\ell}$, we have $w_kr^{\ell}
  \in M^-$.  In this case $c_k^{\ell}$ is the required direct path.
  Otherwise, $(k,\ell)\not\in T^+(M^-)$, so $w_kr^{\ell}\in S\setminus
  M^-$. Let $k'=j^+(w_kr^{\ell'}) = k-\ell'$.  We
  have $w_kr^{\ell'}=x^im_{k'}$ for some $i \geq 0$.  Thus
  $m_{k'}r^{\ell-\ell'}x^i = w_kr^{\ell}$, so since $w_kr^{\ell} \in S \setminus 
   M^-$, the same is true for $m_{k'}r^{\ell-\ell'}$, and $\ell -
  \ell' \leq k'$.  Since $\ell'<\ell \leq k$, we have $k'>1$ and
  $\ell'+1\leq k$.  This means that $w_kr^{\ell'+1}\in S$, and
  $m_{k'}r\in S\setminus M^-$, as otherwise we would have
  $m_{k'}x^ir=w_kr^{\ell'+1}\in M^-$, so $(k,\ell'+1)$ would be in
  $T^+(M^-)$. By induction there is a direct path $c_P$ from $m_{k'}$
  of length $\ell-\ell'$, so $c_k^{\ell'}c_P$ is a direct path of
  length $\ell$ from $m_k$.
\end{proof}

\subsection{The structure of the Macaulay
  matrix} \label{sec:MacaulayMatrix}

For the rest of this section, we fix the standard $(1,1)$ grading on
$S=K[x,y]$, and a Hilbert function
$h:\mathbb{Z}_{\ge0}\to\mathbb{Z}_{\ge0}.$ Let $I \subseteq
K[c_i^{\ell} : (i,\ell) \in T^+(M^-)][x,y]$ be the ideal of the universal family over 
$C_{\prec}(M^-)$ as in Theorem~\ref{t:evainuniversal}.  Note that there
are $\qq{d}=d+1$ monomials of degree $d$ in $S$.

For any $d\ge0$, and for any basis of $I_d$, we may write the coefficients of the basis as
the columns of a matrix $R$ with entries in $K[C_\prec(M^-)]$; such a
matrix is called a degree-$d$ \emph{Macaulay matrix} for $I$, and has
size $(d+1) \times (d+1- h(d))$.  For any collection $E$ of $h(d)$
monomials of degree $d$, there is a polynomial in $I$ with support in
$E$ if and only if the minor
indexed by rows corresponding to monomials not in $E$
is zero.  The matroid
$\mathcal{M}(I_d)$ is thus exactly characterized by which
maximal minors of $R$ vanish.

We begin by choosing a basis for
$I_d$ via the combinatorial set-up given in Section
\ref{sec:ArrowsAndPaths}.
For each of the $d+1-h(d)$ monomials
$m\in M^-_d$, the polynomial
$g_m:=\frac{m}{m_{j^-(m)}}f_{j^-(m)}\in I_d$ has initial term $m$ with
respect to $x\prec y$.

As the polynomials
$\{g_m\}_{m\in M^-_d}$ have distinct initial terms, they are all
linearly independent. Since $\dim(I_d)=d+1-h(d)=\abs{M^-_d},$ we
conclude that $\{g_m\}_{m\in M^-_d}$ is a basis for $I_d.$

Let $R$ be the matrix with columns the coefficient vectors of the
polynomials $g_m$.  This is a degree-$d$ Macaulay matrix for $I$.  We
index the rows by $\QQ{d}$ in increasing order with respect to $x\prec
y$, and index the columns by the monomials in $M^-_d$ in increasing
order with respect to $x\prec y$, so $R_{m',m}$ is the coefficient of
$m'$ in $g_m.$

We now make a series of observations about the matrix $R$.

\begin{property}\label{rem:UpperTriangular}
  The matrix $R$ is upper triangular in the following sense. If
  $m'\succ m,$ then $R_{m',m}=0$ as $g_m$ has initial term $m$.
  Since $M^-$ is the lexicographic ideal with respect to $x \prec y$ when $(a,b)=(1,1)$,
  the monomials $m\in M^-_d$ are
  $\prec$-consecutive; this means that the entries $R_{m,m}$ comprise
  a diagonal
  of $R$.   We
  conclude that all entries below this diagonal are zero. The entries
  along this diagonal are all 1, corresponding to the fact that $m_k$
  has coefficient 1 in $f_k$. See Figure \ref{fig:MacaulayMatrixShape}.
\end{property}

\begin{figure}
    \centering
    \begin{align*}
      \begin{array}{ccccccccccc}
          &\multicolumn{10}{c}{\multirow{9}{*}{
          $\begin{pmatrix}
            *&\textcolor{red}{\times}&\cdots&\textcolor{red}{\times}&\textcolor{red}{\times}\\
            *&*&\cdots&\textcolor{red}{\times}&\textcolor{red}{\times}\\
            \vdots&\vdots&\ddots&\vdots&\vdots\\
            *&*&\cdots&*&\textcolor{red}{\times}\\
            1&*&\cdots&*&*\\
            0&1&\cdots&*&*\\
            \vdots&\vdots&\ddots&\vdots&\vdots\\
            0&0&\cdots&1&*\\
            0&0&\cdots&0&1\\
          \end{pmatrix}$
                               }}\\
        &\\
        &\\
          &\\
        {\raisebox{-5pt}{\tiny $m^*\rightarrow$}}&\\
        &\\
        &\\
        &\\
         &\\
         &\\
      \end{array}
    \end{align*}
    \caption{The structure of the Macaulay matrix $R$. In the
      base-change $\bar{R}$ to the ring $K[C_\prec(M^-)]/\langle{Y}\rangle$, the
      $\textcolor{red}{\times}$ entries are zero.}
    \label{fig:MacaulayMatrixShape}
  \end{figure}
\begin{property}\label{rem:CoeffsArePaths}
  Theorem \ref{t:CoefficientsArePaths} gives a combinatorial
  description of the entry $R_{m',m}$. Namely, let $\ell$ be such that
  $m'=mr^\ell.$ Then
  $$R_{mr^\ell,m}=\sum_{\substack{\text{$P$ a path from}\\\text{$m_{j^-(m)}$
        of length $\ell$}}}c_P.$$
\end{property}

\begin{property}\label{rem:UpperArrowsToZero}
  Let $m^*$ be the smallest monomial in $M_d^-$ with respect to $x
  \prec y$.  It will be convenient to consider the entries of $R$ in a
        {\em quotient} of $K[C_\prec(M^-)]$ where some variables
  have been set to zero.  Let
  $Y=\{ c_i^{\ell} : (i,\ell)\in T^+(M^-), \thickspace \thickspace i>j^-(m^*)\}.$ Let
  $\bar{R}$ be the base-change of the matrix $R$ to
  $K[C_\prec(M^-)]/\langle{Y}\rangle$. That is, $\bar{R}$ is a
  Macaulay matrix for the universal ideal over the \emph{coordinate
    subspace} defined by $\langle Y \rangle$ in $C_\prec(M^-)\cong\A^{\abs{T^+(M^-)}}.$

  The reason for using this quotient is as follows. Suppose
  $m\in M^-_d$,
  so $m\succeq m^*$. Then $j^-(m)\ge j^-(m^*).$ By Definition
  \ref{d:Path}, a nonempty path $P$ from $m_{j^-(m)}$ either (1) contains an
  element of $Y$, in which case
  $c_P=0\in K[C_\prec(M^-)]/\langle{Y}\rangle$, or (2) is a path from
  $m_{j^-(m^*)}.$ It follows that
$\bar{R}$ has entries
  $$\bar{R}_{mr^\ell,m}=\sum_{\substack{\text{$P$ a path
        from}\\\text{$m_{j^-(m^*)}$ of length $\ell$}}}c_P.$$ In
  particular, 
  $\bar{R}$ is \emph{lower} triangular in the following sense.  Note that for 
  $\ell_0=j^-(m^*)$ we have that $m_{j^-(m^*)}r^{\ell_0}$ is a power of
  $x$. Then for $\ell'>\ell_0,$ we have
  $\bar{R}_{mr^{\ell'},m}=0$. This implies the vanishing of all
  entries of $\bar{R}$ that lie above the main diagonal; see Figure
  \ref{fig:MacaulayMatrixShape}. Furthermore, it follows from
  Proposition \ref{p:coeffsoffsubk} and Property~\ref{rem:CoeffsArePaths}
  that all entries on the main diagonal are
  nonzero.
\end{property}

\begin{property}\label{rem:MatrixGrading}
  Define a grading on $K[C_\prec(M^-)]$ by $\deg(c_i^\ell)=\ell.$ Then
   $R_{mr^\ell,m}$ is homogeneous of degree $\ell$. In particular,
  the degree is constant along diagonals of $R$, and satisfies
  $\deg(R_{mr,m'})=\deg(R_{m,m'r^{-1}})=\deg(R_{m,m'})+1.$ This also
  implies that for every square submatrix $R'$ of $R$, the minor
  $\det(R')$ is a homogeneous polynomial. Additionally, for any
  \emph{maximal} square submatrix $R'$ the degrees of the diagonal
  entries of $R'$ are a nonincreasing sequence (read starting at the
  top left as usual); this follows from the fact that $R'$ is obtained
  by deleting only rows (and no columns) from $R$. The same holds for
  $\bar{R}$.
\end{property}

\begin{example}
  Consider the monomial ideal $M^-=\langle
  x^6,x^4y,x^2y^2,xy^3,y^4\rangle$. We have $T^-(M^-)=\emptyset.$ The
  chosen degree-4 Macaulay matrix for $C_{\prec}(M^-)$ is
  \begin{align*}
    R=\kbordermatrix{&x^2y^2&xy^3&y^4\\
    x^4&c_1^1 c_2^1+c_2^2 & c_1^1 c_3^2 & c_1^1 c_4^3 \\
      x^3y&c_1^1+c_2^1 & c_1^1 c_2^1+c_2^2+c_3^2 & c_1^1 c_3^2+c_4^3 \\
      x^2y^2&1 & c_1^1+c_2^1 & c_1^1 c_2^1+c_2^2+c_3^2 \\
      xy^3&0 & 1 & c_1^1+c_2^1 \\
    y^4&0 & 0 & 1 }
         .
  \end{align*}
  The base-change to $K[C_\prec(M^-)]/\langle{Y}\rangle$ is
  \begin{align*}
    \bar{R}=
    \kbordermatrix{&x^2y^2&xy^3&y^4\\
    x^4&c_1^1 c_2^1+c_2^2 & 0 & 0 \\
      x^3y&c_1^1+c_2^1 & c_1^1 c_2^1+c_2^2 & 0 \\
      x^2y^2&1 & c_1^1+c_2^1 & c_1^1 c_2^1+c_2^2 \\
      xy^3&0 & 1 & c_1^1+c_2^1 \\
    y^4&0 & 0 & 1 }
         .
  \end{align*}
  
  Observe how the various properties above apply to these matrices:
  \begin{itemize}
  \item Both are upper triangular in the sense of Property
    \ref{rem:UpperTriangular}:  there are zeros below the
    diagonal $m=m'$.
  \item The matrix $\bar{R}$ is lower triangular in the sense of
    Property \ref{rem:UpperArrowsToZero}.
  \item The homogeneity of Property \ref{rem:MatrixGrading} is
    satisfied with $\deg(c_1^1)=\deg(c_2^1)=1$,
    $\deg(c_2^2)=\deg(c_3^2)=2$, and $\deg(c_4^3)=3$.
  \end{itemize}
\end{example}

\subsection{Proof of Theorem \ref{t:tropicalideal}}\label{sec:GenericUniformity}

We now prove:
\begin{theorem}\label{t:MacaulayMatrixMinors}
  Fix the standard $(1,1)$-grading on $S$, and a Hilbert function $h$, and fix $d\ge0$ such that $h(d)<d+1.$
  Then every $(d+1-h(d)) \times (d+1-h(d))$ minor of $R$ is a
  nonzero polynomial in $K[C_\prec(M^-)]$.
\end{theorem}

\begin{proof}
  For convenience, in this proof let $n_0=d+1-h(d)>0$.  As in
  Property \ref{rem:UpperArrowsToZero}, we define $m^*$ to be the
  smallest monomial (with respect to $x\prec y$) in $M^-_d$.  Again as
  in Property \ref{rem:UpperArrowsToZero}, we work with the Macaulay
  matrix $\bar{R}$ over $K[C_\prec(M^-)]/\langle{Y}\rangle$; if a
  minor is nonzero in this ring, it is also nonzero in
  $K[C_\prec(M^-)]$.
  
  Fix an $n_0\times n_0$ submatrix $R'$ of $\bar{R}$.  By Property
  \ref{rem:UpperArrowsToZero}, the $(i,j)$th entry $R'_{i,j}$ of $R'$ is
  of the form
  $$\sum_{\substack{\text{$P$ a path from}\\\text{$m_{j^-(m^*)}$ of
        length $\ell_{i,j}$}}}c_P$$ for some $\{\ell_{i,j}\}_{1\le
    i,j\le n_0}$.  Note that the sum is zero if $\ell_{i,j}<0$.  We have
  $\ell_{j,j} \geq 0$ for all $1 \leq j \leq n_0$ by
  Property~\ref{rem:UpperTriangular}.
  
  The chosen minor is then
  \begin{align}\label{eq:Minor}
    \det(R')=\sum_{\sigma\in S_{n_0}}(-1)^{\sgn(\sigma)}\prod_{j=1}^{n_0}\sum_{\substack{\text{$P$ a path from}\\\text{$m_{j^-(m^*)}$ of
    length $\ell_{j,\sigma(j)}$}}}c_P.
  \end{align}

  By Proposition \ref{p:coeffsoffsubk}, the path
  $p_{j^-(m^*),\ell_{j,j}}$ exists for all $1\le j\le n_0$. Hence we may
  define:
  $$Q=\prod_{j=1}^{n_0}c_{p_{j^-(m^*),\ell_{j,j}}}.$$ 
  Then $Q$ is a monomial in $K[C_\prec(M^-)]/\langle{Y}\rangle$,
  and $Q$ appears as a term of the right side of \eqref{eq:Minor} when
  $\sigma=\id$. 
  We will
  show that in fact, $Q$ appears with coefficient 1 in $\det(R')$,
  with the only contribution coming from that term.

  \begin{claim} \label{cl:lexicographic1} Suppose
    $C = \prod_{j=1}^s c_{p_{k,\ell_i}}$, with
    $\ell_1 \geq \ell_2 \geq \dots \geq \ell_s \geq 0$, and we also
    have $C = \prod_{i=1}^s c_{P_i}$, where each $P_i$ is a path from
    $m_k$, and $l(P_1) \geq l(P_2) \geq \dots l(P_s) \geq 0$.  Then we
    have the following inequality with respect to the lexicographic
    order on $\mathbb Z^s$:
    $$(\ell_1,\dots,\ell_s) \succeq (l(P_1),\dots, l(P_s)).$$
  \end{claim}

  \begin{proof}[Proof of Claim \ref{cl:lexicographic1}]
If $C=1$ then all paths have length zero, and the claim follows.  We
now assume that $C \neq 1$.  The proof is by induction on $s$.  The
base case is $s=1$, in which case we must have $P_1= p_{k,\ell_1}$, so
the inequality is an equality.  Suppose now that $s>1$, and the result
is true for smaller values of $s$.  Recall from
Definition~\ref{Def:DirectPath} that every variable $c_i^{\ell}$
dividing $C$ has $i$ occurring in some $z_n$ as defined there.  Let
$c_i^{\ell'}$ be the variable with $i$ minimal dividing
$c_{P_1}$. Since $c_{P_1}$ divides $C$, we have $c_i^{\ell'}$ dividing
$c_{p_{k,\ell_j}}$ for some $1 \leq j \leq s$.  We claim that the
length of the part of $p_{k,\ell_j}$ before the step $c_i^{\ell'}$ is
at least as long as the part of the path $P_1$ before $c_i^{\ell'}$,
so $\ell_j \geq l(P_1)$, with equality only if $p_{k,\ell_j} =
P_1$. To see this, note that the part of $P_1$ before $c_i^{\ell'}$
contains only variables $c_{i'}^\ell$ where $i'<i$ is the index of
some $z_n$, while the part of $p_{k, \ell_j}$ before $c_i^{\ell'}$
contains every $z_n$ with $i'<i$.  Since the length $\ell$ of
$c_{i'}^{\ell}$ is at most the length of the associated $z_n$, we have
$\ell_j \geq l(P_1)$, and so $\ell_1 \geq l(P_1)$, with equality only
if $j=1$ and $p_{k,\ell_1} = P_1$.  When the inequality is strict we
have the strict inequality $(\ell_1,\dots,\ell_s) \succ
(l(P_1),\dots,l(P_s))$, while otherwise the induction hypothesis
applied to $C/c_{p_{k,\ell_1}}$ yields the desired inequality.
  \end{proof}

  \begin{claim}\label{cl:lexicographic2}
    For $\sigma\in S_{n_0},$ let $\Pi(\sigma)$ denote the integer
    partition
    $\ell_{1,\sigma(1)}+\ell_{2,\sigma(2)}+\cdots+\ell_{n_0,\sigma(n_0)}.$
    We treat $\Pi(\sigma)$ as a nonincreasing list of integers, whose
    sum is the degree of $\det(R')$ with respect to the grading in
    Property \ref{rem:MatrixGrading}. Then for all
    $\sigma\in S_{n_0},$ we have $\Pi(\sigma)\succeq\Pi(\id)$ with
    respect to the lexicographic order on $\mathbb{Z}_{\geq 0}^{n_0},$ with
    equality only if $\sigma=\id$.
  \end{claim}
  \begin{proof}[Proof of Claim \ref{cl:lexicographic2}]
    Suppose $\sigma\ne\id$ and $\Pi(\sigma)\preceq\Pi(\id).$ Let
    $i\in\{1,\ldots,n_0\}$ be minimal such that $\sigma(i)>i.$ Then
    $\ell_{1,1},\ldots,\ell_{i-1,i-1}$ are parts of both $\Pi(\sigma)$
    and $\Pi(\id).$ By Property \ref{rem:MatrixGrading}, $\ell_{j,j}$
    is nonincreasing as $j$ increases, so
    $\ell_{1,1},\ldots,\ell_{i-1,i-1}$ are the $i-1$ largest parts of
    $\Pi(\id)$. Since $\Pi(\sigma)\preceq\Pi(\id),$ we must have that
    $\ell_{1,1},\ldots,\ell_{i-1,i-1}$ are the $i-1$ largest parts of
    $\Pi(\sigma)$. The next largest part of $\Pi(\id)$ is
    $\ell_{i,i},$ but we know that $\ell_{i,\sigma(i)}>\ell_{i,i}$,
    since $\sigma(i)>i$ and, by Property
    \ref{rem:MatrixGrading}, $\ell_{i,j}$ strictly increases as $j$
    increases. This contradicts $\Pi(\sigma)\preceq\Pi(\id).$
  \end{proof}

  Claims \ref{cl:lexicographic1} and \ref{cl:lexicographic2} together
  show that in the sum \eqref{eq:Minor}, the monomial $Q$ appears only
  in the term $\sigma=\id,$ which is the product
  \begin{align}\label{eq:DiagonalProduct}
    \prod_{j=1}^{n_0}R'_{j,j}&=\prod_{j=1}^{n_0}\sum_{\substack{\text{$P$ a path from}\\\text{$m_{j^-(m^*)}$ of
    length $\ell_{j,j}$}}}c_P
  \end{align}
  Finally, we argue that the coefficient of $Q$ in
  \eqref{eq:DiagonalProduct} is 1. Order the variables $c_i^{\ell}$ so
  that $c_i^{\ell} \succ c_j^{\ell'}$ if $i>j$ or $i=j$ and
  $\ell>\ell'$. Then $c_{p_{j^-(m^*),\ell_{j,j}}}$ is the largest
  monomial $c_P$ in the resulting lexicographic order, when $P$ varies
  over all paths from $m_{j^-(m^*)}$ of length $\ell_{j,j}$. The initial
  term of $R'_{j,j}$ is thus $c_{p_{j^-(m^*),\ell_{j,j}}}$ with
  coefficient 1. The initial term of the product
  \eqref{eq:DiagonalProduct} is the product of the initial terms,
  namely $Q$. Thus $Q$ appears in $\det(R')$ with coefficient 1, so we
  conclude that $\det(R')$ is a nonzero element of
  $K[C_\prec(M^-)]$, for any field $K$.
\end{proof}

Theorem~\ref{t:MacaulayMatrixMinors} is the key to proving Theorem~\ref{t:tropicalideal}.

\begin{proof}[Proof of Theorem~\ref{t:tropicalideal}]
  For an ideal $I\in\Hilb_S^h$, we have
  $\mathcal{M}(I_d)=U_{h(d),d+1}$ for all $d\ge0$ if and
  only all maximal minors of all Macaulay matrices for $I$ are nonzero in degrees where $h(d)>0$.
  By Theorem \ref{t:MacaulayMatrixMinors}, these minors are {\em nonzero} polynomials in $K[C_{\prec}(M)]$, so
    $\mathcal{M}(I_d)=U_{h(d),d+1}$ for all $d\ge0$ if and only if 
  $I$ is in the complement of the vanishing sets of these finitely many
 polynomials.
 The set of such $I$ forms a nonempty open subset of $C_{\prec}(M^-),$
 and hence of $\Hilb_S^h.$ This implies the main claim of Theorem
 \ref{t:tropicalideal}; the second claim in Theorem
 \ref{t:tropicalideal} follows from the standard fact that if $K$ is
 an infinite field, then any nonempty open subset of $\mathbb{A}^n_K$
 contains a $K$-point.
\end{proof}

\begin{remark}
The proofs of Theorems \ref{t:spineedge} and \ref{t:tropicalideal} rely on $K$ being infinite, but we do not know of a counterexample to either one with $K$ finite.
\end{remark}

\begin{remark}\label{rem:InitialRowReduction}
  If $I\in\Hilb_S^h$ satisfies the conclusion of Theorem
  \ref{t:tropicalideal}, then necessarily $\inn_{\prec}(I)=M^-$ and
  $\inn_{\prec^{opp}}(I)=M^+$; that is, $I\in E(M^-,M^+).$ Indeed, for $I\in C_{\prec}(M^-),$
  $\inn_{\prec}(I)_d$ is the span of the monomials corresponding to
  leading ones in the reduced column-echelon form of $R$. Thus the
  condition $\inn_{\prec^{opp}}(I)=M^+$ is equivalent to the
  nonvanishing of a \emph{single} maximal minor of $R$.
  This is a strictly weaker condition than the nonvanishing of \emph{all}
  maximal minors, 
  as guaranteed by Theorem \ref{t:tropicalideal}.
\end{remark}
\begin{remark}
Theorem \ref{t:tropicalideal} determines $\mathcal{M}(I_d)$ when $I$ is a general element of $E(M^-,M^+)$. It is natural to ask if the theorem can be generalized to determine the matroid of a general element of an arbitrary edge-scheme $E(M,M')$, at least when $E(M,M')$ is irreducible. In small examples, even when $E(M,M')$ is irreducible, $\mathcal{M}(I_d)$ is often non-uniform. For example, this occurs in $N=6$, in the edge in Figure \ref{fig:spine6} connecting $(4,1,1)$ and $(3,1,1,1),$ where $\mathcal{M}(I_2)$ has $xy$ as a loop. 
\end{remark}

\subsection{Discussion of other gradings} \label{sec:Nonstandard}

In this section we show that Theorem~\ref{t:tropicalideal} does not
hold for all edges in the spine, so the standard-graded hypothesis
is necessary.

We first note that the degree-$d$ matroid can have {\em loops} and
{\em coloops} in degrees where the entire matroid is not trivial.

\begin{example}
  Let $(a,b)=(2,3)$. Then the two monomial ideals $M^-=(x^7,xy,y^4)$
  and $M^+=(x^6,xy,y^5)$ share a Hilbert function $h$. (Here $N=10$.) The ideal $M^-$
  has the unique positive significant arrow $(2,2)$, and the universal
  ideal $I$ over $C_{\prec}(M^-)$ is thus
  $\langle x^7,xy,y^4+c_2^2x^6\rangle.$

  The degree-12 Macaulay matrix is
  $$\kbordermatrix{& x^3y^2 & y^4\\ x^6 & 0 & c_2^2 \\ x^3y^2 & 1 & 0
    \\ y^4 & 0 & 1 \\}.$$ The matroid $\mathcal M(I_{12})$ on
  ground set $\{x^6,x^3y^2,y^4\}$ has circuits
  $\{\{x^3y^2\},\{x^6,y^4\}\}.$ In particular, $\mathcal M(I_{12})$ is not a
  uniform matroid, due to the existence of the loop $x^3y^2.$ This
  loop is forced to exist 
since $h(5)=0,$ so $xy\in I$, and thus $x^3y^2 \in I$,  for any ideal $I$ with Hilbert function $h$.

  Furthermore, the degree-8 Macaulay matrix is
  $$\kbordermatrix{& xy^2 \\ x^4 & 0 \\ xy^2 & 1 \\},$$ so the matroid
  $\mathcal M(I_8)$ on ground set $\{x^4,xy^2\}$ has the unique
  circuit $\{xy^2\}.$ Again, $\mathcal M(I_8)$ is not a uniform
  matroid.  In addition to the loop $xy^2$, there is also the coloop
  $x^4,$ which is forced to exist by the structure of $h$. To see
  this, note that since $xy\in I$ as noted above, we have $xy^2\in I$.
  As $h(8)=1,$ we must have $x^4 \not \in I$ for any ideal $I$ with Hilbert function $h$, so the matroid $\mathcal M(I_8)$ has a coloop.
\end{example}

We now see, however, that loops and coloops do not entirely account
for the failure of Theorem \ref{t:tropicalideal}.

\begin{example}\label{ex:MajorCounterexample}
  Let $(a,b) = (2,3)$.  Let $h$ be the Hilbert function of the
  monomial ideal $M^- = \langle x^{10}, x^7y, x^2y^3,xy^5,y^6
  \rangle$. Then $M^+=\langle x^9, x^5y, x^4y^3,xy^5,y^7 \rangle$. (Here $N=29.$) The ideal $M^-$ has the positive significant arrows
  $$T^+(M^-) = \{ (2,1), (4,2), (4,3) \}.$$
  Thus the universal ideal $I$ over $C_{\prec}(M^-)$ is
  $$\langle x^{10},x^7y, x^2y^3+c_2^1 x^5y, xy^5+c_2^1x^4y^3, y^6+c_2^1x^3y^4+c_4^2x^6y^2+c_4^3x^9 \rangle.$$
  The degree-$18$ Macaulay matrix is 
  $$\kbordermatrix{& x^3y^4 & y^6\\ x^9 &  0 & c_4^3 \\ x^6y^2 & c_2^1 & c_4^2 \\ x^3y^4 & 1 & c_2^1 \\ y^6 & 0 & 1
    \\ }.$$   The degree-$18$ matroid of $I$ thus has rank $2$ on the ground set $\{x^9, x^6y^2,  x^3y^4, y^6 \}$, with circuits
  $$\{ \{x^3y^4, x^6y^2\}, \{x^9, x^3y^4, y^6\}, \{x^9, x^6y^2, y^6\} \}.$$
  This is not the uniform matroid, and does not have any loops or
  coloops. This is the smallest example we know in which a matroid appears that is not the direct sum of a uniform matroid with a collection of loops and coloops.
\end{example}
\begin{remark}
  In Example \ref{ex:MajorCounterexample}, 
  $\trop(I)$ is ``maximally general'', in the following sense. Let
  $J \subseteq \mathbb B[x,y]$ be any tropical ideal with Hilbert function $h$, in
  the sense of \cite{TropicalIdeals}. Then for all $d\ge0,$ the
  matroid $\mathcal M(J_d)$ is a weak image of $\trop(I)_d.$
\end{remark}

\end{document}